\documentclass[3p,,preprint,9pt]{elsarticle}
\usepackage{graphicx}
\usepackage{amsfonts,amsmath,amssymb,amsbsy,enumerate}

\usepackage[normalem]{ulem}
\usepackage[mathscr]{euscript}
\usepackage{pstricks}
\usepackage{multirow}
\usepackage{verbatim}

\usepackage{subfigure,psfrag}
\usepackage{algorithm}
\usepackage[noend]{algpseudocode}
\usepackage{lineno}
\graphicspath{{./Figuras/}} %%%%

\newcommand{\CP}{\textsc{CP}} 
\newcommand{\NP}{\text{\textsc{NP}}}

\newcommand{\bigO}{\mathcal{O}}

\let\And\undefined
\let\Or\undefined

\newcommand{\pe}{\mathrm{pe}}

\newcommand{\ecc}{\mathrm{ecc}}

\newcommand{\rad}{\mathrm{rad}}

\DeclareMathOperator*{\argmax}{arg\,max} % Jan Hlavacek

\algrenewcommand\algorithmicthen{:}
\algrenewcommand\algorithmicdo{:}
\algnewcommand\And{\textbf{ and }}
\algnewcommand\Or{\textbf{ or }}
\algnewcommand\New{\textbf{ new }}
\algnewcommand\To{\textbf{ to }}
\algnewcommand\DownTo{\textbf{ down to }}
\algnewcommand\Break{\textbf{break}}
\algnewcommand\Null{\textbf{null}}
\algnewcommand\True{\textbf{true}}
\algdef{SE}[SUBALG]{Indent}{EndIndent}{}{\algorithmicend\ }%
\algtext*{Indent}
\algtext*{EndIndent}
\newtheorem{theorem}{Theorem}
\newtheorem{lemma}{Lemma}
\newtheorem{proposition}{Proposition}
\newtheorem{question}{Question}
\newtheorem{corollary}{Corollary}
\newtheorem{claim}{Claim}
\newtheorem{property}{Property}
\newproof{proof}{Proof}

\def\lastname{G\'omez, Guti\'errez}
\begin{document}
	
	\title{Path eccentricity of graphs} 
	\author[1]{Renzo G\'omez\corref{cor1}\fnref{fn1}}
	\ead{rgomez@ic.unicamp.br}
	\author[2]{Juan Guti\'errez\fnref{fn2}}
	\ead{jgutierreza@utec.edu.pe}

	\cortext[cor1]{Corresponding author}
	\fntext[fn1]{This research has been partially supported by FAPESP - S\~ao Paulo Research Foundation (Proc.~2019/14471-1).}
	
	\fntext[fn2]{This research has been partially supported by Movilizaciones para Investigación AmSud, PLANarity and distance
IN Graph theory. E070-2021-01-Nro.6997.}

	\affiliation[1]{organization={Institute of Computing, University of Campinas},
					city={Campinas}, country={Brazil}}
	\affiliation[2]{organization={Departamento de Ciencia de la Computaci\'on, Universidad de Ingenier\'ia y Tecnolog\'ia (UTEC)},
					city={Lima}, country={Peru}}

\begin{abstract}
Let $G$ be a connected graph. The eccentricity of a path $P$, denoted 
by $\ecc_G(P)$, is the maximum distance from $P$ to any vertex in $G$. 
In the \textsc{Central path} ({\CP}) problem our aim is to find a path 
of minimum eccentricity. This problem was introduced by Cockayne et al., 
in 1981, in the study of different centrality measures on graphs. They 
showed that {\CP} can be solved in linear time in trees, but it is known 
to be $\NP$-hard in many classes of graphs such as chordal bipartite graphs, 
planar 3-connected graphs, split graphs, etc.

We investigate the path eccentricity of a connected graph~$G$ as a parameter. 
Let $\pe(G)$ denote the value of $\ecc_G(P)$ for a central path $P$ 
of~$G$. We obtain tight upper bounds for $\pe(G)$ in some graph 
classes. We show that $\pe(G) \leq 1$ on biconvex graphs and 
that $\pe(G) \leq 2$ on bipartite convex graphs. Moreover, we 
design algorithms that find such a path in linear time. 
On the other hand, by investigating the longest paths of a graph, 
we obtain tight upper bounds for $\pe(G)$ on general graphs 
and $k$-connected graphs. 

Finally, we study the relation between a central path and a longest 
path in a graph. We show that on trees, and bipartite permutation 
graphs, a longest path is also a central path. Furthermore, for superclasses 
of these graphs, we exhibit counterexamples for this property.
\end{abstract}
\begin{keyword}
  central path \sep path eccentricity \sep biconvex graph \sep tree \sep cactus \sep $k$-connected graph \sep longest path.
\end{keyword}

\maketitle

%\linenumbers

\section{Introduction and Preliminaries}\label{sec:intro}

For the problem considered here, the input graph is always 
connected (even if it is not stated explicitly). The \textit{length} of 
a path is its number of edges. The \textit{distance} 
between two vertices $u$ and $v$ in a graph $G$, denoted by~$d_G(u, v)$, 
is the minimum length of a path between them. Moreover, the distance 
between a vertex $u$ and a set $S \subseteq V(G)$ 
is~$d_G(u, S) = \min\{ d_G(u, v) : v \in S \}$. 

In the single facility location problem, given a network, one seeks a best 
site to place a facility in order to serve the other sites of the network. 
In some applications, we are interested in finding a location that 
minimizes the maximum distance to the sites. A concept in Graph Theory, 
related to this measure, is the eccentricity. Let $G$ be a graph. 
The \textit{eccentricity} of a vertex $u$, in $G$, 
is $\max\{ d_G(u, v) : v \in V(G) \}$. A vertex of minimum eccentricity 
is called a \textit{center} of $G$. Moreover, the \textit{radius} of $G$, 
denoted by $\rad(G)$, is the eccentricity of a center. 
%Clearly, 
%we can find a center of a graph in polynomial time. From a 
%theoretical point-of-view, another line of research is to 
%consider the eccentricity as a graph parameter. Let $n$ denote 
%the order of a graph $G$. In 1987, Erd\H{o}s et al.~\cite{ErdosPPT89} 
%showed that the radius of $G$ is at most $(3n - 5)/(2\delta(G) + 1) + 5$. If $G$ 
%is $3$-connected, Harant~\cite{Harant93} showed an upper bound of $n/4 + 8$ 
%for the radius of $G$. 

A classical generalization for the single facility location problem 
is the $k$-center problem. In this case, we are interested in finding a 
set $S$ of $k$ sites, inside a network $G$, that 
minimizes $\max\{ d_G(u, S) : u \in V(G) \}$. Observe that the places 
we choose, for those facilities, do not obey any particular structure. 
When the facilities represent bus stops or train stations, a natural 
restriction to impose would be the existence of a route (link) between 
a bus stop and the next one. Thus, the facilities induce a path in the network.
This problem is called the {\scshape Central path} 
(\CP) problem. More formally, given a graph~$G$, we want to find a path $P$  
that minimizes
\[
	\ecc_G(P) := \max\{ d_G(u, V(P)) : u \in V(G) \}.
\]
A path that attains this minimum is called a \textit{central path} 
of $G$. Cockayne et al.~\cite{CockayneHH81} investigated the minimal (regarding 
vertex-set inclusion) central paths of a graph, and designed a linear-time 
algorithm for finding such a path on trees. Slater~\cite{Slater82} studied central paths of minimum 
length, and also showed a linear-time algorithm on trees. Observe that a 
Hamiltonian path is a central path of zero eccentricity. Thus, {\CP} 
is $\NP$-hard on the classes of graphs for which the Hamiltonian path 
problem is $\NP$-complete. This is the case for cubic planar $3$-connected 
graphs~\cite{GareyJT76}, split graphs, chordal bipartite 
graphs~\cite{Muller96}, etc.

%More recently, the {\scshape Minimum Eccentricity Shortest Path} 
%({\scshape MESP}) problem~\cite{DraganL17} has also been studied. 
%In this variant of the {\CP} problem, we are interested 
%in a shortest induced path of minimum eccentricity. 
%Furthermore, Birmel\'{e} et al.~\cite{BirmeleMP16} designed 
%a linear $3$-approximation algorithm for {\scshape MESP}.

%Our objective is to study of the
A parameter that is naturally associated with central paths 
is the \textit{path eccentricity} of a graph $G$, denoted by $\pe(G)$, 
and defined as 
\[
	\pe(G) := \min\{ ecc_G(P): P \text{ is a path in } G\}.
\]
We observe that, in a graph $G$, a path $P$ such that $\ecc_G(P) \leq \ell$ is 
also called an \textit{$\ell$-dominating path} in the literature. Obviously, $G$ has a 
$\ell$-dominating path if and only if 
$\pe(G) \leq \ell$.
%Now, we describe the organization of this text. 
%First, in Section~\ref{sec:prelim} we present the main concepts and 
%definitions that will be used throughout the text.
%% Furthermore, 
%we present results related to {\CP} and the parameter $\pe(G)$ that 
%%exists in the literature, and derive additional implications for the 
%%{\CP} problem.

%In Section~\ref{sec:bcg}, we study the central paths 
We start by studying the path eccentricity  
of convex and biconvex graphs in Section~\ref{sec:bcg}.
Corneil 
et. al.~\cite{CorneilOL97} showed that if $G$ is an \textit{asteroidal triple-free} (AT-free)
graph, then $\pe(G) \leq 1$. 
This class of graphs includes
interval graphs, permutation graphs, among others.
Furthermore, they showed a linear-time 
algorithm to find such a path. Therefore, in this class of graphs, {\CP} 
has the same computational complexity as the \textsc{Hamiltonian path} problem. 
The latter problem is still open on AT-free graphs, however it can be 
solved in polynomial time in some of its subclasses. Keil~\cite{Keil85} 
showed a linear-time algorithm for the Hamiltonian path problem on interval 
graphs, and Spinrad et.\ al.~\cite{SpinradBS87} showed and analogous 
result for bipartite permutation graphs. Therefore, 
the \textsc{Central path} problem can be solved in linear-time 
on interval graphs and bipartite permutation graphs.  
%we have the 
%following result. 

%\begin{corollary} \label{cor:classes}
%	The \textsc{Central path} problem can be solved in linear-time 
%	on interval graphs and bipartite permutation graphs.  
%\end{corollary}

We extend these results, by studying superclasses of bipartite 
permutation graphs, such as convex and biconvex graphs, 
obtaining the following results.

\begin{itemize}
\item We can find a 2-dominating path in a convex graph in linear time (Theorem~\ref{thm:convex-alg}).
This implies that if $G$ is a convex graph, then $\pe(G) \leq 2$ (Corollary~\ref{cor:convexbound}). 
Moreover, this bound is tight.
\item We can find a dominating path in a biconvex graph in linear time (Theorem~\ref{thm:dompath-biconvex}). 
This implies that if $G$ is a biconvex graph, then $\pe(G) \leq 1$ (Corollary~\ref{cor:biconvexbound}). 
Moreover, this bound is tight.
\end{itemize}

In Section~\ref{sec:kconn}, we investigate  
upper bounds for $\pe(G)$ on general graphs and on $k$-connected graphs.
%the path centers 
%of $k$-connected graphs.
Since a single vertex can be regarded as a path of length zero, 
any upper bound for the radius of a graph is also valid 
for $\pe(G)$. 
%From a 
%theoretical point-of-view, another line of research is to 
%consider the eccentricity as a graph parameter.
Let $n$ denote 
the order of a graph $G$. In 1987, Erd\H{o}s et al.~\cite{ErdosPPT89} 
showed that the radius of $G$ is at most $(3n - 5)/(2\delta(G) + 1) + 5$. If $G$ 
is $3$-connected, Harant~\cite{Harant93} showed an upper bound of $n/4 + 8$ 
for the radius of $G$. 
Hence these are valid bounds for $\pe(G)$ for connected and 3-connected graphs.
We improve these results by proving the following.

\begin{itemize}
\item For every connected graph $G$ on $n$ vertices, we have $\pe(G) \leq \frac{n-1}{3}$. Moreover, this bound is tight (Theorem \ref{th:k=1}).
\item For every $k$-connected graph $G$ on $n$ vertices, with $k\geq 2$, we have
$\pe(G) \leq \frac{n+k}{3k+2}$. Moreover, this bound is tight (Theorem \ref{theo:k-conn-2}).
\end{itemize}

The previous tight upper bounds for $\pe(G)$
were obtaining by studying the eccentricity of longest paths. 
So, a natural question is whether a longest path is always a central path.
We investigate the relationship between central paths and longest 
paths in Section~\ref{sec:longest}. We show that

\begin{itemize}
\item In trees, $4$-connected planar graphs,
proper interval graphs, and
bipartite permutation graphs, every longest path 
is a central path (Corollaries \ref{cor:longest-tree}, \ref{cor:longest-properinterval}, and \ref{cor:longest-bpg}).
\item On the other hand, on cacti, $2$-connected planar graphs, interval graphs and convex graphs, 
it is not always the case that every longest path is a central path
(Figures \ref{fig:planar}, \ref{fig:counterInt}, and \ref{fig:cx})
.
\end{itemize}

%Furthermore, we show that these results are, 
%in some sense, sharp. For this, in each case, we show a graph class 
%that contains it, and a familiy of graphs where the difference 
%between $\pe(G)$ and the eccentricity of every longest path is 
%arbitrarily large. 
Finally, we discuss some open problems and 
give concluding remarks in Section~\ref{sec:remarks}.

To conclude this section, we give a simple observation, 
used throughout the text, to derive families of graphs 
that exhibit a tight upper bound for $\pe(G)$. 
Let $G$ be a graph. Let us denote by $\mathcal{C}(G)$ 
the collection of components in $G$. 

\begin{proposition} \label{prop:certificate}
	Let $G$ be a connected graph. Let $S \subseteq V(G)$ such 
	that $|\mathcal{C}(G - S)| \geq |S| + 2$. Then,  
	\[
		\pe(G) \geq \min\{ \ecc_G( S, C ) : C \in \mathcal{C}(G - S) \},
	\]
	where $\ecc_G(S, C) := \max\{ d_G(u, S) : u \in V(C) \}$.
\end{proposition}
\begin{proof}
	Let $P$ be a central path of $G$. Note that $P - S$ has at most $|S| + 1$ 
	components. As $|\mathcal{C}(G - S)| \geq |S| + 2$, 
	there exists a component $C^* \in \mathcal{C}(G - S)$ such 
	that $V(P) \cap V(C^*) = \emptyset$. Thus, any path 
	between~$V(P)$ and~$V(C^*)$ contains a vertex in $S$. 
	This implies that 
	\[
		\pe(G) = \ecc_G(P) \geq d_G( u, V(P) ) \geq d_G(u, S) + d_G(S, V(P)) \geq d_G(u, S), 
	\]
	for any vertex $u \in V(C^*)$. Therefore,~$\pe(G) \geq \ecc(S, C^*)$.
	\qed
\end{proof}

%\section{Preliminaries and related results} \label{sec:prelim}

%\input{prelim.tex}

%\section{Interval graphs} \label{sec:interval}

%\input{interval.tex}

\section{Path eccentricity of convex graphs and subclasses} \label{sec:bcg}

As mentioned in the previous section, bipartite permutation graphs have bounded 
path eccentricity. In this section, we seek to extend this result studying classes 
of graphs, such as convex and biconvex graphs, that properly contain the bipartite 
permutation graphs. Moreover, we show that these classes also have bounded path 
eccentricity. 

\subsection{Convex graphs}

Let $G = (X \cup Y, E)$ be a bipartite graph. 
We say that $G$ is \textit{$X$-convex} if there 
is an ordering of the vertices in~$X$, such that $N_G(y)$ 
consists of consecutive vertices in 
that ordering, for each $y \in Y$. We call such 
ordering a \textit{convex ordering} of $X$. 
If $G$ is an $X$-convex graph, we abbreviate this fact saying 
that $G$ is convex. In what follows, $G$ denotes 
a convex graph, and we consider 
that $\{ x_1, x_2, \ldots, x_n \}$ is a convex ordering 
of~$X$. Moreover, we say that $x_i < x_j$ if $i < j$; and 
for each vertex~$y \in Y$, we define 
\[
	\begin{array}{rcl}
		\ell(y) & := & \min\{ x : x \in N_G(y) \}, \text{ and} \\
		r(y) & := & \max\{ x : x \in N_G(y) \}.  
	\end{array}
\]

Algorithm~\ref{alg:2dom-path-convex} finds a $2$-dominating 
path, say $P$, of $G$ greedily. Starting with $P = \langle a_1 \rangle$, 
where $a_1 = x_1$, we extend $P$ by appending a vertex $y$, in $N_G(a_1)$, 
that maximizes $r(y)$. After that, we also append $a_2 = r(y)$ to $P$ and 
repeat this procedure until the last vertex in the path is $x_n$.

\begin{algorithm}
  \caption{2Dom-Path-Convex$(G)$}
  \label{alg:2dom-path-convex}
  Input: A connected convex graph $G = (X \cup Y, E)$\\
  Output: A $2$-dominating path of $G$
  
\begin{algorithmic}[1]
\State Find a convex ordering of $X$, say $\{ x_1, x_2, \ldots, x_n \}$ \label{lst:ord1}
\State $a_1 \gets x_1$
\State $k \gets 1$
\While{$a_k \neq x_n$} \label{lst:loop-conv} 
	\State $b_{k} \gets \argmax\{ r(y) : y \in N_G(a_k) \}$
	\State $a_{k + 1} \gets r( b_k )$ \label{lst:max}
	\State $k \gets k + 1$
\EndWhile \label{lst:loop-end-conv}
\State \Return $\langle a_1, b_1, \ldots, b_{k - 1}, a_k \rangle$
\end{algorithmic}
\end{algorithm}

\begin{theorem} \label{thm:convex-alg}
	Algorithm~\ref{alg:2dom-path-convex} finds a $2$-dominating 
	path in a convex graph $G$ in $\bigO(|V(G)| + |E(G)|)$ time.
\end{theorem}
\begin{proof}
	First, we show that the loop of 
	lines~\ref{lst:loop-conv}-\ref{lst:loop-end-conv} 
	terminates. For this, we show that if $a_k \neq x_n$ (at any iteration) 
	then 
	\[
		a_k < \max\{ r(y) : y \in N_G(a_k) \}. 
	\] 
	That is, at line~\ref{lst:max}, we have $a_{k} < a_{k + 1}$.
	Let $X_1 = \{ x : x \in X,\, x \leq a_k \}$ and $Y_1 = N_G(X_1)$.
	Let $X_2$ (resp. $Y_2$) be $X \setminus X_1$ (resp. $Y \setminus Y_1$). 
	Let~$G_1$ (resp. $G_2$) be the graph induced by $X_1 \cup Y_1$ 
	(resp. $X_2 \cup Y_2$). Since $G$ is connected and $G_2$ contains $x_n$, 
	then there must exist an edge linking $G_1$ to $G_2$. Now, 
	by the definition of~$Y_1$, there is no edge between $X_1$ and $Y_2$. 
	Thus, the edge linking $G_1$ to $G_2$ must be between a 
	vertex $y^* \in Y_1$ and vertex in $X_2$. The convexity of $X$ 
	implies that $\ell(y^*) \leq a_k < r(y^*)$, and thus~$y^* \in N_G(a_k)$. 
	Therefore, we have $a_{k} < r(y^*) \leq a_{k + 1}$.
	
	Let $k^*$ be the value of $k$ at the end of the algorithm. 
	By the previous arguments $a_{k^*} = x_n$. Now, we show 
	that $P = \langle a_1, b_1, \ldots, b_{k^* - 1}, a_{k^*} \rangle$ is 
	a $2$-dominating path of $G$. Since $b_i$ is adjacent to $a_i$ 
	and $a_{i + 1}$, convexity implies that 
	\[
		\displaystyle \bigcup_{i = 1}^{k^* - 1}{N_G(b_i)} = X.
	\]
	Thus, it suffices to show that every vertex in $Y$ is at distance at 
	most two from a vertex in $P$. Let $y \in Y$. Since $a_1 = x_1$ 
	and $a_{k^*} = x_n$, then $a_i \leq \ell(y) \leq a_{i + 1}$, for 
	some $1 \leq i < k^*$. Moreover, as $b_i$ is adjacent to $a_i$ 
	and $a_{i+1}$, the vertex $b_i$ must be adjacent to $\ell(y)$ by convexity. 
	Thus, $y$ is at distance at most two from $P$. Regarding the time complexity 
	of the algorithm, the ordering at line~\ref{lst:ord1} is
	computed in linear time~\cite{BoothL76}. 
	Finally, the loop of lines~\ref{lst:loop-conv}-\ref{lst:loop-end-conv} 
	runs in $\bigO(|V(G)|+|E(G)|)$ time. Therefore, 
	Algorithm~\ref{alg:2dom-path-convex} runs in~$\bigO(|V(G)|+|E(G)|)$ time.
	\qed
\end{proof}

The previous theorem implies the following upper bound on $\pe(G)$.

\begin{corollary} \label{cor:convexbound}	
	If $G$ is a convex graph, then $\pe(G) \leq 2$.
\end{corollary}

Now, we show that this bound is tight. Furthermore, we show 
that for every positive integer $k$, there exists a $k$-connected convex 
graph that attains this bound. Let $k$ be a positive integer, 
and let $\ell := k(k + 2)$. Let $K_{k, \ell}$ be the complete bipartite 
graph with one side of size $k$ and the other side of size $\ell$. 
Since $\ell > k$, then $K_{k, \ell}$ is a $k$-connected graph. 
In what follows, we consider that $K_{k, \ell}$ is a $(X, Y)$-bipartite 
graph such that $|X|= \ell$, $|Y| = k$ and 
\[
	X = \{ x_1, x_2, \ldots, x_\ell \}.
\]
We observe that $K_{k, \ell}$ is $X$-convex. 
Let $G$ be the graph defined as follows:
\[
	\begin{array}{rcl}
		V(G) & = & V(K_{k, \ell}) \cup \{ z_i : 1 \leq i \leq k + 2\}, \\
		E(G) & = & E(K_{k, \ell}) \cup \{ z_ix_j : 1 \leq i \leq k + 2, \, (i - 1) \cdot k + 1 \leq j \leq i \cdot k \}.
	\end{array}
\]

Since each vertex $z_i$ is adjacent to $k$ vertices in $X$, and $K_{k, \ell}$ is 
bipartite $k$-connected, we have that $G$ is also bipartite $k$-connected. 
Furthermore, as $N_G(z_i) = \{ x_j \in X : (i - 1) \cdot k + 1 \leq j \leq i \cdot k \}$, 
$G$ is also $X$-convex. Finally, if we consider $S = Y$, then each component 
$C$ in $\mathcal{C}(G - S)$ is a star such that its center is the unique vertex, in $C$, 
not adjacent to $S$. Thus, $\ecc_G(S, C) = 2$, for each $C \in \mathcal{C}(G - S)$. 
By Proposition~\ref{prop:certificate} and Corollary~\ref{cor:convexbound}, we 
have that $\pe(G) = 2$. We show an example of this construction for $k = 2$ in 
Figure~\ref{fig:conv-bound}.

\begin{figure}
	\centering
	\includegraphics[scale=.7]{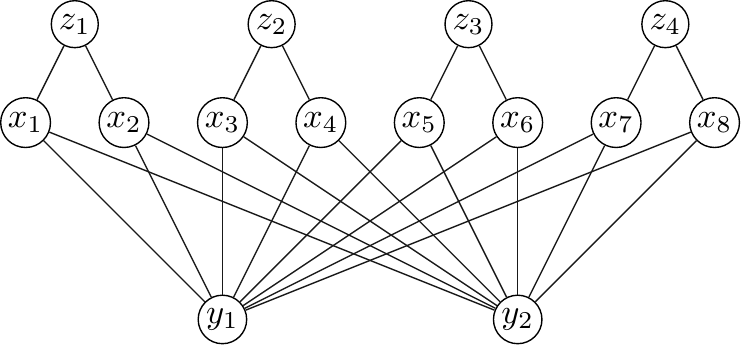}
	\caption{A convex $2$-connected graph $G$ such that $\pe(G) = 2$.}
	\label{fig:conv-bound}
\end{figure}

\subsection{Biconvex graphs}

In this section we study the parameter $\pe(G)$ on biconvex graphs 
(a subclass of convex graphs), and improve the result obtained in 
the previous section. We show that $\pe(G) \leq 1$ for every biconvex 
graph $G$. We say that a $(X, Y)$-bipartite graph is \textit{biconvex} if 
it is both, $X$-convex and $Y$-convex. Moreover, we say that $V(G) = X \cup Y$ 
admits a \textit{biconvex ordering}. A subclass of biconvex graphs 
are the bipartite permutation graphs. A result of Corneil et.\ 
al.~\cite{CorneilOL97} implies that bipartite permutation graphs 
admit a dominating path. For ease of explanation, we will show an 
algorithm that finds such a path on bipartite permutation graphs. 
After that, we modify this algorithm in order to find a dominating path in 
any biconvex graph $G$ (implying that $\pe(G) \leq 1$). 

Let $G = ( X \cup Y, E )$ be a connected bipartite permutation 
graph. Spinrad et al.~\cite{SpinradBS87} (see Theorem~1) showed that $V(G)$ admits a 
biconvex ordering, say $X = \{ x_1, \ldots, x_n \}$ 
and~$Y = \{y_1, \ldots, y_m\}$, that satisfy one additional property. 
Let $u$ and $v$ be two vertices in the same part of $G$. If $u$ comes 
before $v$ in the previous ordering, we write~$u < v$ to indicate this 
fact. Consider $x_i, x_k \in X$ and $y_a, y_c \in Y$ such that $x_i \leq x_k$ 
and $y_a \leq y_c$. Then, we have the following: 
\begin{enumerate}[$(\star)$]
	\item if $x_iy_c$ and $x_ky_a$ belong to $E$, then $x_iy_a$ and $x_ky_c$ also belong to $E$. \label{prop:cross}
\end{enumerate}

An ordering of $V(G)$ that satisfies~$(\star)$ is called 
\textit{strong}. Yu \& Chen~\cite{YuC95} (see Lemma~7) noted that 
every strong ordering of $V(G)$ is also biconvex.
The main tool that we use to construct a dominating path in $G$ is 
a decomposition of $V(G)$ showed by Uehara \& Valiente~\cite{UeharaV07}. 
They showed that, given a strong ordering of $V(G)$, we can 
decompose $G$ into subgraphs $(K_1 \cup J_1), \ldots, (K_k \cup J_k)$ 
in the following way. First, let $K_1$ be the graph induced 
by $N(x_1) \cup N(y_1)$. Let $J_1$ be the set of isolated 
vertices (maybe empty) in $G - K_1$. Now, consider 
that~$G = G - (K_1 \cup J_1)$, and repeat this process until 
the graph becomes empty. They called this decomposition 
a \textit{complete bipartite decomposition} of $G$. Moreover, 
they showed that it can be obtained in~$\bigO(|V(G)|)$. 
Let~$\ell_X(K_i)$ and $r_X(K_i)$ be the vertices in $V(K_i) \cap X$ 
such that
\[
	\begin{array}{rl}
		\ell_X(K_i) \leq x, \text{ for each } x \in V(K_i) \cap X, \\
		r_X(K_i) \geq x, \text{ for each } x \in V(K_i) \cap X. \\
	\end{array}
\] 
In a similar way, we define $\ell_Y(K_i)$ and $r_Y(K_i)$, for $i = 1, \ldots, k$.
Observe that, 
\[
	V(K_i) = N_G( \ell_X(K_i) ) \cup N_G( \ell_Y(K_i) ), \, \text{for } i = 1, \ldots, k.
\]
In what follows, we say that two sets of vertices are \textit{adjacent} if there exists 
an edge linking them. Uehara \& Valiente~\cite{UeharaV07} showed that 
this decomposition satisfies the following properties:

\begin{enumerate}[$a)$]
	\item $K_i$ is a nontrivial (with at least one edge) complete bipartite graph. Moreover, 
	$V(K_i)$ is adjacent to~$V(K_{i + 1})$, but not to $V(K_{\ell})$, for $\ell > i + 1$. \label{prop:comp}
	\item For each $i \in \{1, \ldots, k\}$, either $J_i \subseteq X$ or $J_i \subseteq Y$. \label{prop:class}
	\item For each $i \in \{1, \ldots, k\}$, if $J_i \subseteq X$ (resp. $J_i \subseteq Y$), then $J_i$ is a continuous 
	sequence of vertices in the ordering of~$X$ (resp. $Y$). Furthermore, $r_X(K_i)\leq a \leq \ell_X(K_{i+1})$ 
	(resp.~$r_Y(K_i)\leq a \leq \ell_Y(K_{i+1})$), for each~$a \in J_i$. \label{prop:Ji}
	\item $N_G(J_i) \subseteq V(K_i)$, for $1 \leq i \leq k$. Furthermore, 
	if $a, a' \in J_i$ and $a \leq a'$, then $N_G(a') \subseteq N_G(a)$. \label{prop:inK}
\end{enumerate}

Let $e$ be an edge with one end in $K_i$ and the other in $K_{i + 1}$. 
We say that $e$ is an \textit{$XY$-edge} if the end of $e$, in $K_i$, belongs to $X$. 
In an analogous way, we define a $YX$-edge. The idea of our algorithm is to 
construct a dominating path of $G$ incrementally. Given a dominating path $P$
of the graph induced by $(K_1 \cup J_1) \cup \ldots \cup (K_i \cup J_i)$, at the $i$-th 
step, we extend $P$ to dominate $(K_{i + 1} \cup J_{i + 1})$. In particular, 
the following three properties will be used to show the correctness of our 
algorithm.

\begin{enumerate}[$a)$]
  	\setcounter{enumi}{4}
	\item If $K_i$ is linked to $K_{i+1}$ with an $XY$-edge (resp. a $YX$-edge), 
	then $r_X(K_i)$ (resp. $r_Y(K_i)$) is adjacent to~$\ell_Y(K_{i+1})$ 
	(resp. $\ell_X(K_{i+1})$). \label{prop:extra1}
	\item If $J_i \subseteq Y$ (resp. $J_i \subseteq X$), 
	then~$r_X(K_i) \in N(v)$ (resp. $r_Y(K_i) \in N(v)$), for 
	every $v \in J_i$. \label{prop:extra2}
	\item If $K_i$ is linked to $K_{i+1}$ by an $XY$-edge and a $YX$-edge, 
	then $J_i = \emptyset$. \label{prop:extra3}
\end{enumerate}

We observe that~$\ref{prop:extra1})$ follows from the convexity of the 
orderings of $X$ and $Y$. Moreover, both~$\ref{prop:extra2})$ 
and~$\ref{prop:extra3})$ follow from~$\ref{prop:Ji})$ and~$\ref{prop:inK})$ 
using the fact that the vertex ordering is biconvex and strong. 
Let $P = \langle a_1, \ldots, a_n \rangle$ and $Q = \langle b_1, \ldots, b_m \rangle$ be 
vertex-disjoint paths, in $G$, such that $a_n$ is adjacent to $b_1$.  
Then, the \textit{concatenation} of $P$ and $Q$, denoted by $P \cdot Q$, is 
the path $\langle a_1, \ldots, a_n, b_1, \ldots, b_m \rangle$. Now, we 
present Algorithm~\ref{alg:dom-path} that finds a dominating path in bipartite 
permutation graphs.

\begin{algorithm}
  \caption{Dom-Path-BPG$(G)$}
  \label{alg:dom-path}
  Input: A connected bipartite permutation graph $G = (X \cup Y, E)$\\
  Output: A dominating path $P$ of $G$
  
\begin{algorithmic}[1]
\State Find a strong convex ordering of $X$ and $Y$, say $\{ x_1, \ldots, x_n \}$ and $\{ y_1, \ldots, y_m \}$ \label{lst:ord2}
\State Let $(K_1, J_1), \ldots, (K_k, J_k)$ be a complete bipartite decomposition of $G$ \label{lst:decomp}
\State $P \gets \emptyset$
\State $v \gets \ell_X(K_1)$ \label{lst:beg1}
\For{$i \gets 1$ \To {$k - 1$}} \label{lst:for}
	\If {$K_i$ is linked to $K_{i+1}$ by an $XY$-edge}
		\State Let $P'$ be a path between $v$ and $r_X(K_i)$ \label{lst:p1}
		\State $v \gets \ell_Y(K_{i + 1})$ \label{lst:beg2}
	\Else
		\State Let $P'$ be a path between $v$ and $r_Y(K_i)$ \label{lst:p2}	
		\State $v \gets \ell_X(K_{i + 1})$ \label{lst:beg3}
	\EndIf
	\State $P \gets P \cdot P'$
\EndFor	\label{lst:endfor}
\If {$J_k \subseteq Y$}	
	\State Let $P'$ be a path between $v$ and $r_X(K_k)$ \label{lst:p3}
\Else
	\State Let $P'$ be a path between $v$ and $r_Y(K_k)$ \label{lst:p4}
\EndIf
\State $P \gets P \cdot P'$
\State \Return $P$
\end{algorithmic}
\end{algorithm}

\begin{theorem} \label{thm:dompath}
	Algorithm~\ref{alg:dom-path} finds a dominating 
	path in a bipartite permutation graph $G$ in $\bigO(|V(G)| + |E(G)|)$.
\end{theorem}
\begin{proof}
	First, we show that the algorithm correctly constructs a 
	path $P$. Note that lines~\ref{lst:beg1}, \ref{lst:beg2}, and~\ref{lst:beg3} 
	imply that the variable $v$ represents either $\ell_X(K_i)$ or~$\ell_Y(K_i)$ 
	at the beginning of each iteration of the loop at 
	lines~\ref{lst:for}-\ref{lst:endfor}. Since $K_i$ 
	is complete bipartite, there is always a path between $v$ 
	and $r_X(K_i)$ or $r_Y(K_i)$. Furthermore, by~$\ref{prop:extra1}$), 
	$r_X(K_i)$ (resp. $r_Y(K_i)$) is adjacent to $\ell_Y(K_{i+1})$ 
	(resp.\ $\ell_X(K_{i+1})$) when there exists an $XY$-edge (resp.\  
	a $YX$-edge) between $K_{i}$ and $K_{i + 1}$. Therefore, Algorithm~\ref{alg:dom-path} 
	returns a path.

	Now, we show that $P$ is a dominating path. First, we prove that $P$ 
	dominates $V(K_i)$, for $i=1, \ldots, k$. Observe that, if $K_i \cong K_{n,m}$ 
	where $n, m \geq 2$, then any path between $v$ and $r_X(K_i)$ 
	or $r_Y(K_i)$ must contain an edge of $K_i$. Thus, $P$ dominates $V(K_{i})$.
	Now, suppose that $K_i \cong K_{1,m}$ for $m \geq 1$.
	Without loss of generality, suppose that $\ell_X(K_i) = r_X(K_i)$ (a center 
	of the star $K_{1,m}$). Observe that, either 
	if~$v = \ell_X(K_i)$ or~$v = \ell_Y(K_i)$, the path $P'$ 
	(at lines~\ref{lst:p1}, \ref{lst:p2}, \ref{lst:p3}, \ref{lst:p4}) 
	will contain a center of $K_{1,m}$. Therefore, $P$ dominates every~$K_i$. 

	Finally, we show that~$P$ dominates 
	every $J_i$, for~$i = 1, \ldots, k$. For this, 
	we show that $P$ contains $r_X(K_i)$ if~$J_i \subseteq Y$, otherwise 
	it contains $r_Y(K_i)$.
	Let $J_i \neq \emptyset$ such that $i < k$. Then, 
	by~$\ref{prop:extra3})$, $K_i$ is linked to~$K_{i+1}$ either by $XY$-edges 
	or $YX$-edges. Without loss of generality, suppose that there exists 
	an $XY$-edge, say~$xy \in E(G)$, between $K_i$ and $K_{i+1}$. 
	Then, line~\ref{lst:p1} implies that $P$ contains $r_X(K_i)$. Thus, 
	by~$\ref{prop:extra2})$, $P$ dominates $J_i$. Finally, lines~\ref{lst:p3} 
	and~\ref{lst:p4} ensure that $P$ dominates $J_k$.

	Regarding the computational complexity of the algorithm, line~\ref{lst:ord2} is
	computed in $\bigO(|V(G)| + |E(G)|)$ time~\cite{SpinradBS87}.
	Moreover, the complete bipartite decomposition of $G$ is computed  
	in $\bigO(|V(G)|)$ time~\cite{UeharaV07}. To conclude, 
	the paths at lines \ref{lst:p1}, \ref{lst:p2}, \ref{lst:p3} and \ref{lst:p4} 
	are computed in constant time, as each $K_i$ is complete bipartite. Therefore, 
	the complexity of Algorithm~\ref{alg:dom-path} is $\bigO(|V(G)| + |E(G)|)$.
	\qed
\end{proof}

\begin{figure}
	\centering
	\includegraphics[scale=.5]{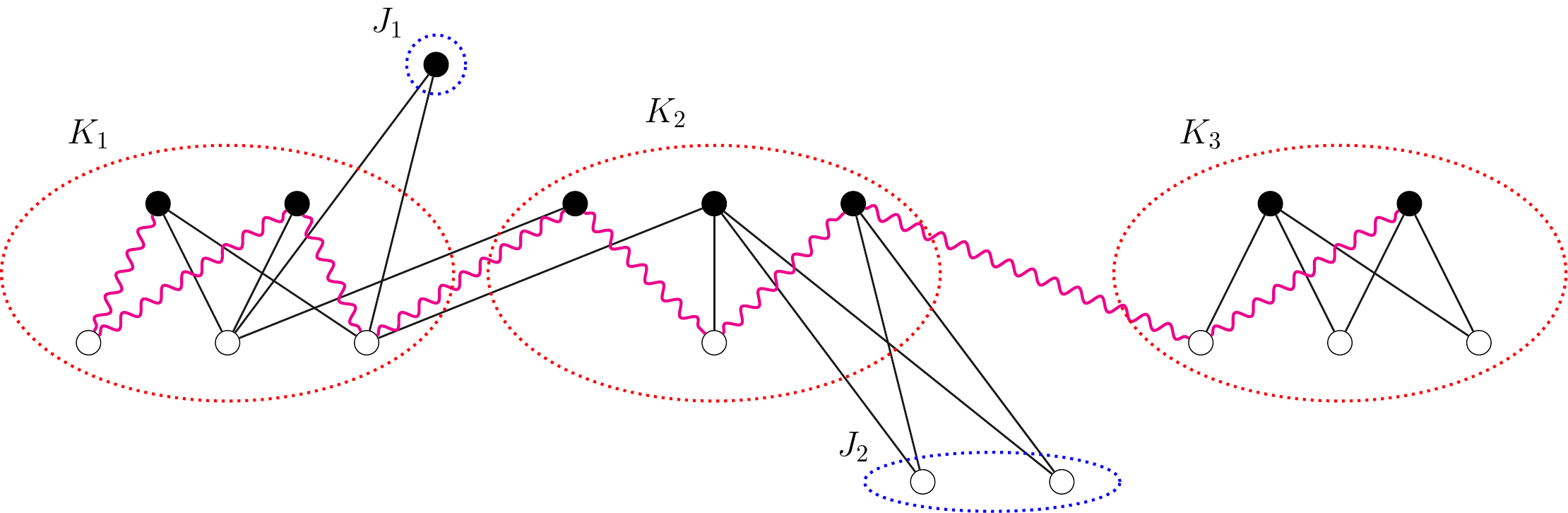}
	\caption{A dominating path found by Algorithm~\ref{alg:dom-path} on a bipartite permutation graph.}
	\label{fig:alg-dom}
\end{figure}

In Figure~\ref{fig:alg-dom}, we show a bipartite permutation graph. The set $X$ 
(resp.~$Y$) is represented by the black (resp.~white) vertices, and a dominating path 
found by Algorithm~\ref{alg:dom-path} is represented by the wavy edges. 
In what follows, we show how to modify Algorithm~\ref{alg:dom-path} 
in order to find a dominating path in a biconvex graph. For this, we use 
a result of Yu \& Chen~\cite{YuC95} which states that every biconvex graph 
contains a bipartite permutation graph as an induced subgraph.
Let $G = (X \cup Y, E)$ be a biconvex graph such that $X = \{ x_1, \ldots, x_n\}$ 
and $Y = \{ y_1, \ldots, y_m\}$ represent a biconvex ordering of $X \cup Y$. 
Let $x_L$ (resp. $x_R$) be the vertex in $N(y_1)$ (resp. $N(y_m)$) 
such that $N(x_L)$ (resp. $N(x_R)$) is not properly contained in any other 
neighborhood set. In case of ties, we choose $x_L$ (resp.~$x_R$) to be 
the smallest (resp.~largest) index vertex. Observe that, we 
may assume that~$x_L \leq x_R$, otherwise we can consider the 
reverse ordering of $X$. Now, let $X_p = \{ x_i : x_L \leq x \leq x_R\}$, 
and let~$G_p = G[X_p \cup Y]$. Yu \& Chen~\cite{YuC95} showed the 
following.

\begin{lemma}[Yu \& Chen, 1995] \label{lem:yu}
	$G_p$ is a bipartite permutation graph.
\end{lemma}

Moreover, Abbas \& Stewart~\cite{AbbasL00} showed further properties 
of $X - X_p$.
	
\begin{lemma}[Abbas \& Stewart, 2000] \label{lem:lorna}
	Let $G = (X \cup Y, E)$ be a biconvex graph. There exists 
	an ordering of $V(G)$, say $X = \{x_1, \ldots, x_n\}$ 
	and $Y = \{ y_1, \ldots, y_m \}$, such that the following hold:
	\begin{enumerate}[$a)$]
		\item $G_p$ is a bipartite permutation graph.
		\item $X_p = \{x_L, x_{L+1}, \ldots, x_R \}$ and $Y = \{ y_1, y_2, \ldots, y_m \}$ is a strong ordering of $V(G_p)$.
		\item For all $x_i$ and $x_j$, where $x_1 \leq x_i < x_j \leq x_L$ 
		or $x_R \leq x_j < x_i \leq x_n$, then $N(x_i) \subseteq N(x_j)$. \label{it:important} 
	\end{enumerate}
\end{lemma}
	
Our aim is to use the above lemma to find a dominating path on a biconvex graph. 
As~$G_p$ is a bipartite permutation graph, it has a complete bipartite decomposition, 
say~$(K_1, J_1), \ldots, (K_k,J_k)$. Furthermore, 
observe that~$x_L = \ell_X(K_1)$ and~$x_R \in \{ r_X(K_k) \} \cup J_k$. 
Thus, by Lemma~\ref{lem:lorna}~$\ref{it:important})$, we have the following.
\begin{corollary} \label{cor:firstlast}
	Let $G_p$ be the bipartite permutation graph defined from $G$, 
	and let $(K_1, J_1), \ldots, (K_k, J_k)$ be its complete bipartite decomposition.
	Then, 
	\[
		\begin{array}{rl}
			N(x_i) \subseteq V(K_1), & \text{for } i = 1, 2, \ldots, L, \\
			N(x_j) \subseteq V(K_k) \cup J_k, & \text{for } j = R, R + 1, \ldots, n. \\
		\end{array}
	\]
\end{corollary}

The previous result implies that, to obtain a dominating path of $G$, we 
need to carefully choose the vertices in $(V(K_1) \cup \{ x_1, \ldots, x_{L-1} \})$ 
and $(V(K_k) \cup J_k \cup \{ x_{R+1}, \ldots, x_n \})$ that belong to the path. 
Algorithm~\ref{alg:dom-path-biconvex} finds a dominating path in a 
biconvex graph. The line~\ref{lst2:loop} of Algorithm~\ref{alg:dom-path-biconvex} 
represents the loop of lines~\ref{lst:for}-\ref{lst:endfor} of 
Algorithm~\ref{alg:dom-path}. We recall the following property of a complete 
bipartite decomposition of $G_p$. We will use it to show the correctness of 
our algorithm.

\begin{enumerate}[$a)$]
  	\setcounter{enumi}{5}
	\item If $J_i \subseteq Y$ (resp. $J_i \subseteq X$), 
	then $r_X(K_i) \in N(v)$ (resp. $r_Y(K_i) \in N(v)$), for 
	every $v \in J_i$.
\end{enumerate}

\begin{algorithm}
	\caption{Dominating-Path-Biconvex$(G)$}
 	\label{alg:dom-path-biconvex}
  	Input: A connected biconvex graph $G = (X \cup Y, E)$\\
  	Output: A dominating path $P$ of $G$
	\begin{algorithmic}[1]
	\State Find a biconvex ordering $\{ x_1, \ldots, x_n \}$ and $\{ y_1, \ldots, y_m \}$ that satisfy Lemma~\ref{lem:lorna}
	\State Let $G_p$ be the graph induced by $\{ x_L, \ldots, x_R \} \cup Y$
	\State Let $(K_1, J_1), \ldots, (K_k, J_k)$ be the complete bipartite decomposition of $G_p$
	\If {$r_Y(K_1) \in N_G(x_1)$}
		\State $P \gets \langle \ell_X(K_1) \rangle$ \label{lst2:start1}
		\State $v \gets r_Y(K_1)$ 
	\Else
		\State $P \gets \emptyset$ \label{lst2:start2}		
		\State $v \gets \min\{ y_j : y_j \in N_G(x_1) \}$		
	\EndIf
	\State Process $(K_1 \cup J_1), \ldots, (K_{k-1} \cup J_{k-1})$ \label{lst2:loop}
	\State Let $P'$ be a path between $v$ and $r_X(K_k)$ \label{lst2:rx}
	\If {$N_G(x_n) \cap J_k \neq \emptyset$}		
		\State Let $u$ be a vertex in $N_G(x_n) \cap J_k$ \label{lst2:u1}
		\State $P' \gets P' \cdot \langle u \rangle$  \label{lst2:jk}
	\Else
		\If {$V(P') \cap N_G(x_n) = \emptyset$}
			\State Let $u$ be a vertex in $N_G(x_n)$ \label{lst2:u2}
			\State $P' \gets P' \cdot \langle u \rangle$  \label{lst2:xn}
		\EndIf
	\EndIf
	\State $P \gets P \cdot P'$ \label{lst2:final}
	\State \Return $P$
\end{algorithmic}
\end{algorithm}

\begin{theorem} \label{thm:dompath-biconvex}
	Algorithm~\ref{alg:dom-path-biconvex} finds a dominating 
	path in a connected biconvex graph $G$.
\end{theorem}
\begin{proof}
	By the proof of Theorem~\ref{thm:dompath}, $P$ dominates $J_1$ and $(K_i \cup J_i)$, 
	for $i = 2, \ldots, k - 1$. In what follows, we show that~$P$ 
	dominates $A = V(K_1) \cup \{ x_1, \ldots, x_{L-1} \}$.
	First, suppose that $r_Y(K_1) \in N_G(x_1)$. By Lemma~\ref{lem:lorna}~\ref{it:important}), 
	the vertex $r_Y(K_1)$ dominates $x_i$, for $x_1 \leq x_i \leq x_{L - 1}$.
	Moreover, since $\{\ell_X(K_1) , r_Y(K_1) \}$ dominates $K_1$, the path~$P$ dominates $A$. 
	Now, suppose that $r_Y(K_1) \notin N_G(x_1)$, and 
	let $y^* = \min\{ y_j : y_j \in N_G(x_1) \}$. Since~$y^* < r_Y(K_1)$, any path 
	from $y^*$ to either $r_X(K_1)$ or $r_Y(K_1)$ must contain a vertex in $V(K_1) \cap X$. 
	Thus, $P$ contains a vertex $x^* \in V(K_1) \cap X$. By similar arguments as before, 
	the set $\{x^*, y^*\} \subseteq V(P)$ dominates $A$.

	Finally, we show that $P$ dominates $B = K_k \cup J_k \cup \{ x_{R+1}, \ldots, x_n \}$.
	Consider the variable $v$ after the execution of line~\ref{lst2:loop}.
	Note that if either $v = \ell_X(K_k)$ or~$v = \ell_Y(K_k)$, the path $P'$ in 
	line~\ref{lst2:rx} dominates $V(K_k)$. Now, we show that $P$ 
	dominates $J_k \cup \{ x_{R+1}, \ldots, x_n \}$. We distinguish two cases. 
	
	\vspace{2mm}
	\noindent \textbf{Case 1:} $N_G(x_n) \cap J_k \neq \emptyset$.

	In this case $J_k \subseteq Y$. By~$\ref{prop:extra2})$, every vertex in $J_k$ 
	is adjacent to $r_X(K_k)$, and thus, $P$ dominates $J_k$. Furthermore, 
	by Lemma~\ref{lem:lorna}~\ref{it:important}), the vertex $u$ (at line~\ref{lst2:u1}) 
	dominates~$x_i$, for $x_{R + 1} \leq x_i \leq x_n$. Therefore, $P$ 
	dominates~$B$.

	\vspace{2mm}
	\noindent \textbf{Case 2:} $N_G(x_n) \cap J_k = \emptyset$. 

	In this case $N_G(x_n) \subseteq V(K_k)$. If $V(P') \cap N_G(x_n) \neq \emptyset$, 
	Lemma~\ref{lem:lorna}~$\ref{it:important})$ implies that $P$ dominates~$B$.
	So, suppose that $V(P') \cap N_G(x_n) = \emptyset$. Then, the vertex 
	$u \in N_G(x_n)$ at line~\ref{lst2:u2} is adjacent to $r_X(K_k)$, 
	and Lemma~\ref{lem:lorna}~$\ref{it:important})$ implies that $P$ 
	dominates $x_i$, for $x_{R + 1} \leq x_i \leq x_{n}$. Now, we show 
	that~$P$ dominates~$J_k$. If~$J_k \subseteq Y$, by~$\ref{prop:extra2})$, 
	the vertex~$r_X(K_k)$ dominates $J_k$. Otherwise, suppose that $J_k \subseteq X$. Let~$x$ be 
	any vertex in~$J_k$, we will show that $u$ dominates~$x$. 
	Observe that the vertices $u$, $r_Y(K_k)$, $x$ and $x_R$ belong to~$G_p$. 
	Moreover, as $u \in V(K_k)$, we have that $u \leq r_Y(K_k)$. By definition of~$G_p$, we also have 
	that~$x \leq x_R$. Observe that, by Lemma~\ref{lem:lorna}~\ref{it:important}), $u$ is adjacent to $x_R$. 
	Furthermore, by~$\ref{prop:extra2})$, the vertex $r_Y(K_k)$ is adjacent to $x$. Thus, 
	since the ordering of $V(G_p)$ is strong, $x$ is adjacent to~$u$. 
	Therefore, $P$ dominates $B$. 
	\qed
\end{proof}

The previous theorem implies the following result.

\begin{corollary} \label{cor:biconvexbound}
	If $G$ is a biconvex graph, then $\pe(G) \leq 1$.
\end{corollary}

To conclude, we observe that this bound is tight even if $G$ 
is $k$-connected. For this, consider the complete bipartite 
graph $K_{k, k + 2}$. Observe that $K_{k, k + 2}$ is $k$-connected 
and also biconvex. Moreover, by Proposition~\ref{prop:certificate}, 
we have~$\pe(K_{k, k + 2}) = 1$.

%\input{bpg.tex}

%\section{Biconvex graphs} \label{sec:biconvex}

%\input{biconvex.tex}

%\section{Convex graphs} \label{sec:convex}

\section{Upper bounds for $\pe(G)$ in general and $k$-connected graphs} \label{sec:kconn}

In this section, we show an upper bound for $\pe(G)$
when $G$ is a general (and connected) graph, and when $G$ is $k$-connected with $k\geq 2$.
Given a path $P$ and two vertices $x$ and $y$ in $P$, we denote by $P[x,y]$ the subpath of $P$ with extremes $x$ and $y$.
Note that $d_P(x,y)=|P[x,y]|$.
We begin by showing a result for any connected graph.
Recall that $\ecc_G(P)= \max \{d_G(u,V(P)): u \in V(G)\}$
and $\pe(G) = \min \{\ecc_G(P):~P \text{ is a path in } G\}$.

\begin{theorem} \label{th:k=1}
For any connected graph $G$ on $n$ vertices, $\pe(G) \leq \frac{n-1}{3}$. Moreover, this bound is tight.
\end{theorem}
\begin{proof}
Let $P$ be a longest path of $G$ with length $L$.
Let $v$ be a vertex of $G$ with $d_G(v,V(P))=\ecc_G(P)$, 
and let $Q$ be a shortest path from $v$ to $P$.
Then, as $|V(P) \cup V(Q)|\leq n$, we have
\begin{equation}\label{eq:Qleqn-L-1}
|Q| \leq n-L-1.
\end{equation}
Also, as $P$ is a longest path, we have
\begin{equation}\label{eq:QleqL2}
|Q| \leq L/2.
\end{equation}
Indeed, otherwise we can join $Q$ with a subpath of $P$ of length at least $L/2$, obtaining a path with length more than $L$, a contradiction.

From \eqref{eq:Qleqn-L-1} and \eqref{eq:QleqL2}, we have
$$
\pe(G) \leq \ecc_G(P) 
= |Q| 
\leq \min \{n-L-1,L/2 \} 
\leq \frac{2}{3}\cdot (n-L-1) +  \frac{1}{3}\cdot \frac{L}{2} 
= \frac{n-1}{3},
$$
where the last inequality follows
because $\min\{a,b\} \leq \alpha\cdot a + (1-\alpha)\cdot b$
for any $a,b \in \mathbb{R}$, with $\alpha \in [0,1]$.
To show that this bound is tight, consider the following family of graphs: start with a 
$K_{1,3}$ and subdivide each edge the same arbitrary number of times.
\qed
\end{proof}

Next we show how to improve
Theorem \ref{th:k=1} when $G$ is $k$-connected and $k\geq 2$.
The proof idea is similar: we will obtain two inequalities (as inequalities (1) and (2) in the proof of Theorem \ref{th:k=1}), and then combine them to obtain the wanted upper bound.
We begin by showing a proposition that is valid for any path $P$.
This is the analogous of equation (1).
% in the proof of Theorem \ref{th:k=1}.

\begin{proposition}\label{prop:eccn-P}
Let $P$ be a path in a $k$-connected graph $G$ with $n$ vertices. Then $\ecc(P)\leq (n-|P|+k-2)/k$.
\end{proposition}
\begin{proof}
For any integer $i \geq 0$, let $S_i=\{v \in V(G):d_G(P,v)=i\}$.
Note that $S_i=\emptyset$ for any $i> \ecc(P)$.
Note also that, as $G$ is $k$-connected,
for any $i$ with $1 \leq i< \ecc(P)$, we have $|S_i|\geq k$.
Indeed, otherwise $S_i$ separates $G$.
Hence, as $S_0=V(P)$, we have
$$
k(\ecc(P)-1) \leq \sum_{i=1}^{\ecc(P)-1}|S_i|\leq n-|V(P)|-1= n-|P|-2,
$$
and the proof follows.
\qed
\end{proof}

The next lemma is the analogous of equation (2) in the proof of Theorem \ref{th:k=1}.
Its proof is given after Theorem \ref{theo:k-conn-2}.

\begin{lemma} \label{lema:eccP2}
Let $P$ be a longest path in a $k$-connected graph $G$. If $k\geq 2$, then $\ecc(P)\leq (|P|+2)/(2k+2)$.
\end{lemma}

With that lemma at hand, we show the main result of this section.

\begin{theorem} \label{theo:k-conn-2}
Let $G$ be a $k$-connected graph of order $n$ such that $k \geq 2$. 
Then $\pe(G)\leq \frac{n+k}{3k+2}$.
\end{theorem}
\begin{proof}
Let $P$ be a longest path of $G$ with length $L$.
By Proposition \ref{prop:eccn-P} and Lemma \ref{lema:eccP2}, we have
\begin{eqnarray*}
\pe(G) &\leq& \ecc(P) \\
 &\leq& \min \left\{\frac{(n-L-2+k)}{k},\frac{L+2}{2k+2}  \right\} \\
&\leq& \frac{kL}{(2k+2+kL)}\cdot \frac{(n-L-2+k)}{k}
+\frac{2k+2}{(2k+2+kL)}\cdot \frac{L+2}{2k+2} \\
&=& \frac{n+k}{3k+2},
\end{eqnarray*}
where the third inequality follows because $\min\{a,b\} \leq \alpha\cdot a + (1-\alpha)\cdot b$
for any $a,b \in \mathbb{R}$, with $\alpha \in [0,1]$.
%and the proof follows.
\qed
\end{proof}

We now proceed to the proof of Lemma \ref{lema:eccP2}.
For this, we use the following proposition, known as ``Fan lemma''.

\begin{proposition}[{\cite[Proposition 9.5]{BondyM08}}]\label{prop:fan-lemma}
 Let~$G$ be a~$k$-connected graph. Let~${v \in V(G)}$
and~$S \subseteq V(G) \setminus \{v\}$.
If~$|S|\geq k$ then there exists a set of~$k$ internally vertex-disjoint paths from $v$ to $S$.
Moreover, every two paths in this set have~$\{v\}$
as their intersection. 
\end{proposition}

\begin{proof}(Of Lemma \ref{lema:eccP2}.)
Let $P$ be a longest path in $G$. Let $L=|P|$.
If $L=n-1$, then $\ecc(P)=0$ and the proof follows.
Otherwise, let $x$ be an arbitrary vertex in $V(G)\setminus V(P)$.
As~$G$ is $k$-connected, we have that $|P|\geq k$ (since any extreme of $P$ has degree at least $k$).
Thus, by Proposition \ref{prop:fan-lemma}, there exists a set $\{Q_1,Q_2,\ldots,Q_k\}$
of internally vertex-disjoint paths from $x$ to $P$.
Let $\{v_1,v_2,\ldots, v_k\}$ be their correspondent extremes in $P$,
and let $p_1$ and $p_2$ be the extremes of $P$.
Moreover, we may assume that 
$d_P(p_1,v_i) < d_P(p_1,v_{i+1})$ holds
for $1 \leq i \leq k-1$.

\begin{claim}
$d_P(p_1,v_1) \geq |Q_1|+|Q_k|-1$ and
$d_P(v_k,p_2) \geq |Q_k|+|Q_1|-1$.
\end{claim}
\begin{proof} (Of Claim 1)
Note that $(P-P[p_1,v_1])\cdot Q_1 \cdot (Q_k-v_k)$ and
$(P-P[v_k,p_2])\cdot Q_k \cdot (Q_1-v_1)$ are paths.
As $P$ is a longest path, we have
$|P|-d_P(p_1,v_1)+|Q_1|+|Q_k|-1\leq |P|,$
which proofs the first part of the claim.
The proof of the second part follows by a similar argument.
\qed
\end{proof}

\begin{claim}
For $1 \leq i \leq k-1$, we have $d_P(v_i,v_{i+1}) \geq |Q_i|+|Q_{i+1}|$.
\end{claim}
\begin{proof} (Of Claim 2)
Note that, for every such $i$, we have that 
$P[v_1,v_i] \cdot Q_i \cdot Q_{i+1} \cdot P[v_{i+1},v_{k}]$ is a path. As
$P$ is a longest path, the proof follows.
\qed
\end{proof}

By Claims 1 and 2, we have

\begin{eqnarray*}
L &=& |E(P)| \\\\
  &=& d_P(p_1,v_1)+ \sum_{i=1}^{k-1}d_P(v_i,v_{i+1})+
  d_P(v_k,p_2) \\
  &\geq& |Q_1|+|Q_k|-1 + \sum_{i=1}^{k-1}(|Q_i|+|Q_{i+1}|)+
  |Q_1|+|Q_k|-1 \\
  &=&2\sum_{i=1}^{k}|Q_i|+|Q_1|+|Q_k|-2.
\end{eqnarray*}
Hence, by pigeonhole principle, there exits a path $Q_i$ with $|Q_i|\leq \frac{L+2}{2k+2}$ and the proof follows.
\qed
\end{proof}

Theorem \ref{theo:k-conn-2} is tight in one sense, as the graph $K_{k,k+2}$ is $k$-connected 
with eccentricity one. Hence, for any~$k \geq 2$, there exists a graph that makes 
Theorem \ref{theo:k-conn-2} tight. We can ask for an stronger result: is it true that, 
for any $\ell \geq 1$ and $k \geq 2$, there exists a graph $G$ such 
that $\pe(G) = \ell = (n + k) / (3k + 2)$? We can answer this question for $k = 2$, 
as the following result shows.

\begin{corollary}
	For each $\ell \geq 1$, there exists a 2-connected graph $G$ such that $\pe(G)= \ell = \frac{n+2}{8}$.
\end{corollary}
\begin{proof}
Let $H$ be a graph isomorphic to $K_{2,4}$ and
let 
\[
	V(H)=\{v_1,v_2,u_1,u_2,u_3,u_4\}
\]
with $v_1$ and $v_2$ in the same side of the bipartition of $H$.
Let $G$ be the graph obtained by subdividing
$\ell - 1$ times every edge of $H$.
Note that, if $G$ has $n$ vertices, then
$\ell = \frac{n + 2}{8}$. Moreover, $G$ is 2-connected.
Now, observe that $G' = G - \{v_1,v_2\}$ has exactly four connected
components which are paths of length $2(\ell - 1)$ with one end adjacent 
to $v_1$ and the other adjacent to $v_2$ in $G$. 
Thus, by Proposition~\ref{prop:certificate}, we have that $\pe(G) \geq \ell$.
Finally, by Theorem~\ref{theo:k-conn-2}, we conclude that~$\pe(G) = \ell = \frac{n+2}{8}$.
\qed
\end{proof}

For $k \geq 3$ and $\ell \geq 2$, we can only give a partial answer to this question. 
To this end, we make the following definition. Let $k$ be an integer, and let $G_1$ 
and $G_2$ be two graphs with at least $k$ vertices. We denote by $G_1 \overset{k}{\equiv} G_2$ the 
following graph, called a \textit{$k$-matching} of $G_1$ and $G_2$.
\[
	\begin{array}{rcl}
		V( G_1 \overset{k}{\equiv} G_2 ) & = & V(G_1) \cup V(G_2), \\ 
		E( G_1 \overset{k}{\equiv} G_2 ) & = & E(G_1) \cup E(G_2) \cup \{ u_iv_i : 1 \leq i \leq k \}, \\ 
	\end{array}
\]
where $u_1, u_2, \ldots, u_k$ (resp. $v_1, v_2, \ldots, v_k$) are distinct vertices 
of $G_1$ (resp. $G_2$). In other words, $G_1 \overset{k}{\equiv} G_2$ is obtained 
from the union of $G_1$ and $G_2$ by linking its vertices by a matching of size $k$. 
Observe that, a $k$-matching of $G_1$ and $G_2$ is not unique in general. Moreover, 
if $G_1$ and $G_2$ are $k$-connected, then $G_1 \overset{k}{\equiv} G_2$ is 
also $k$-connected. When we consider the sequence of graphs $G_1, G_2, \ldots, G_m$, 
the graph $G_1 \overset{k}{\equiv} G_2 \overset{k}{\equiv}  \ldots \overset{k}{\equiv} G_m$ 
is obtained from the union of $G_1, G_2, \ldots, G_m$ by linking $G_i$ to $G_{i+1}$ by 
a matching of size $k$, for $i = 1, \ldots, m - 1$. We show in Figure~\ref{fig:kmatch} an 
example of the graph $K_4 \overset{2}{\equiv} K_3 \overset{2}{\equiv} K_2$.

\begin{figure}
	\centering
	\includegraphics[scale=.5]{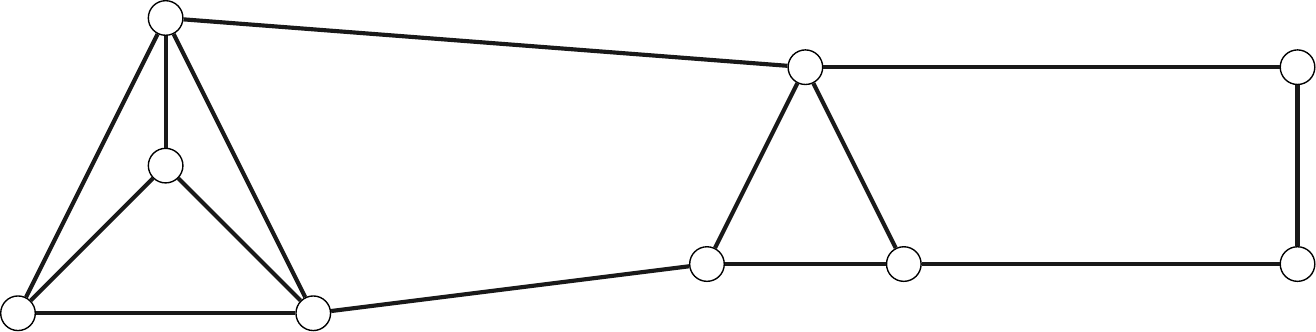}
	\caption{A graph $K_4 \overset{2}{\equiv} K_3 \overset{2}{\equiv} K_2$.}
	\label{fig:kmatch}
\end{figure}

In what follows, we denote by $P^k_\ell$ a graph isomorphic 
to $G_1 \overset{k}{\equiv} G_2 \overset{k}{\equiv}  \ldots \overset{k}{\equiv} G_\ell$, 
where $G_\ell \cong K_{k + 1}$ and~$G_i \cong K_k$, for $1 \leq i < \ell$. Note 
that $P^k_\ell$ can be regarded as a path of length $\ell - 1$, where each vertex represents  
either $K_k$ or $K_{k+1}$, and each edge represents a matching of size $k$. In that 
sense, we will call $G_1$ (resp. $G_{\ell}$) the \textit{tail} (resp. \textit{head}) 
of $P^k_\ell$. Finally, let $H^k_\ell$ be the graph obtained as follows. 
First, consider $k + 2$ copies of $P^{k}_\ell$, say $H_1, \ldots, H_{k+2}$. 
Let $T_i = \{ v^j_i : 1 \leq j \leq k\}$ be the set of vertices in the tail 
of $H_i$, for~$i = 1, \ldots, k + 2$. Moreover, 
let $V^j = \{ v^j_1, v^j_2, \ldots, v^j_{k+2} \}$, for $j = 1, \ldots, k$.
We obtain $H^k_\ell$ by identifying each $V^j$ as a single vertex. 
We call the set of vertices resulting from such procedure, the \textit{base} 
of $H^k_\ell$. We show in Figure~\ref{fig:Hconst} an example of $P^3_2$ and $H^3_2$, 
where the base of $H^3_2$ is depicted by full vertices. 

\begin{figure}[bth]
	\centering
	\includegraphics[scale=.35]{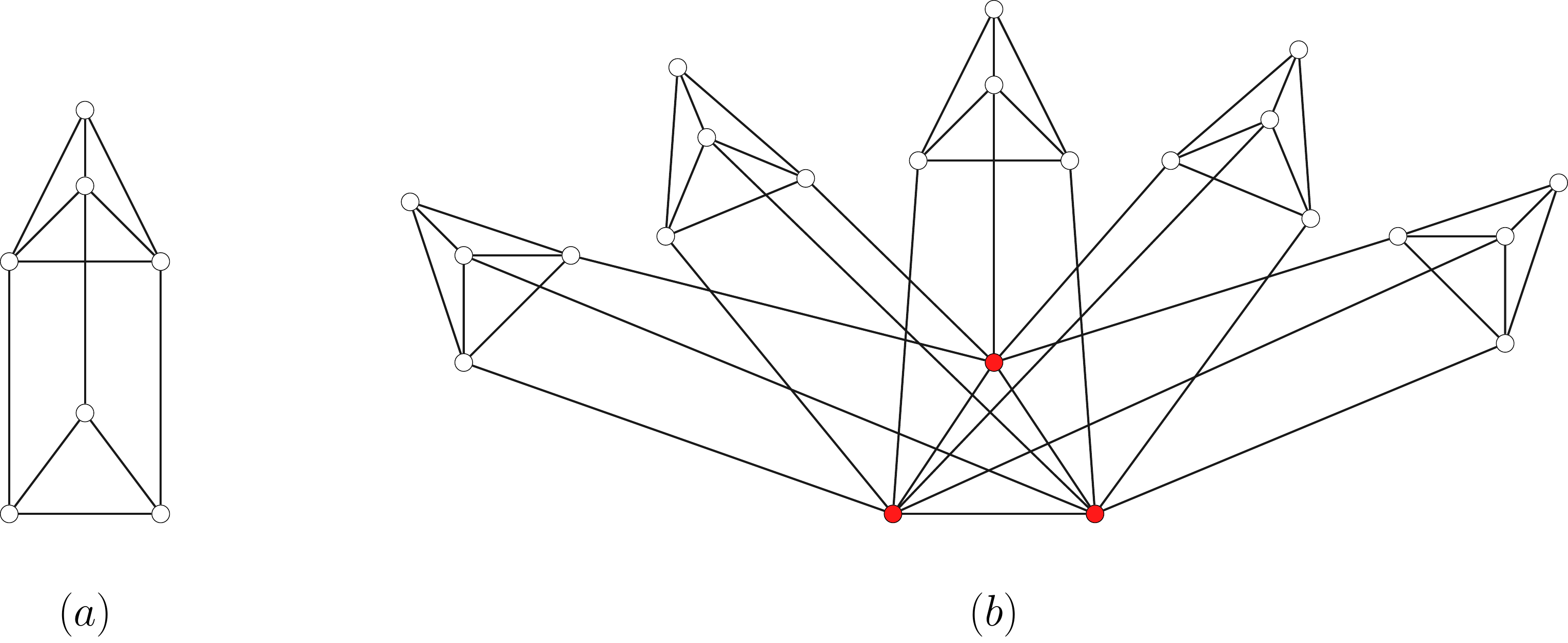}
	\caption{$(a)$ $P^3_2$; and $(b)$ $H^3_2$.}
	\label{fig:Hconst}
\end{figure}

Now, we will show that $\pe(H^k_{\ell}) = \ell$. Let $S \subseteq V(H^k_\ell)$ be 
the base of $H^k_\ell$. Since $S$ induces a complete subgraph, there is a path, 
say $Q$, that contains every vertex of $S$. To find an upper bound for 
the eccentricity of $Q$, first, we study the eccentricity of a vertex $u$ that 
belongs to the tail of $P^k_\ell$. Since the head of $P^k_\ell$ is isomorphic 
to $K_{k+1}$, there must exists a vertex $w$, in the head of $P^k_\ell$, such 
that $d_{P^k_\ell}(u, w) = \ell$, and thus, $\ecc_{P^k_\ell}(u) = \ell$.  
As $S$ is the tail of every induced $P^k_\ell$ in $H^k_\ell$, the previous argument 
implies that~$\pe(H^k_\ell) \leq \ell$. Finally, the inequality $\pe(H^k_\ell) \geq \ell$ 
follows from Proposition~\ref{prop:certificate}.

In what follows, we analyse for which cases $H^k_\ell$ is a tight example.
By Theorem~\ref{theo:k-conn-2}, we must have that
\begin{equation} \label{eq:tight}
	\ell = \pe(H^k_\ell) \leq \frac{|V(H^k_\ell)| + k}{3k + 2} < \ell + 1.
\end{equation}
On the other hand,
\[
	\begin{array}{rcl}
		|V(H^k_\ell)| & = & (k + 2)( (\ell - 1)k + 1 ) + k \\
		\ & = & (k + 2)k(\ell - 1) + 2k + 2. \\
	\end{array}
\]
Replacing the last equality in inequality~(\ref{eq:tight}), and multiplying both 
sides by $(3k + 2)$, we obtain the following $$(k + 2)(\ell - 1)k + 3k + 2 < (\ell + 1)(3k + 2).$$  	  
The last inequality is equivalent to
\begin{equation} \label{eq:final}
	\ell k  + 2(\ell + k) - (\ell - 1)k^2 > 0. 
\end{equation}
Now, we analyse inequality~(\ref{eq:final}) for some values 
of $k$. If $k = 3$, we obtain $-4\ell + 15 > 0$ which 
holds for~$\ell = 1, 2, 3$. Similarly, if $k = 4$, we have 
that $-10 \ell + 24 > 0$ which holds for~$\ell= 1, 2$. Finally, 
for $k \geq 5$, inequality~(\ref{eq:final}) only holds for $\ell = 1$.
Therefore, for $k \geq 3$, there exists a graph $G$ that attains the bound given in 
Theorem~\ref{theo:k-conn-2} in the following cases: $k = 3$ 
and $\ell = 1, 2, 3$, and $k = 4$ and $\ell = 1, 2$, where $\ell = \pe(G)$. 
We observe that $\pe(H^k_\ell)$ is equal to the diameter of $P^k_{\ell}$. Thus, 
it is natural to try improving the previous construction by considering a 
$k$-connected graph with larger diameter. Unfortunately, as noted by 
Caccetta \& Smyth~\cite{CaccettaS92} (see Theorem~1), the graph $P^k_\ell$ 
has the largest diameter among the $k$-connected graphs with the same number 
of vertices.

We conclude this section observing that, if we are interested in dominating paths, 
that is, a path in which every vertex in the graph is at distance at most one, then 
Theorem \ref{theo:k-conn-2} give us the following corollary.

\begin{corollary} \label{cor:k-conn-dom}
If $G$ is a $k$-connected graph with $k>\frac{n-4}{5}$ then $G$ has a dominating path.
\end{corollary}
%\begin{proof}
%By Theorem \ref{theo:k-conn}, $\pe(G) \leq \frac{n+k}{3k+2}<2$.
%\end{proof}

Faudree et. al.~\cite{FaudreeGJW17} showed 
that for a $k$-connected graph $G$, if $\delta(G) > \frac{n-2}{k+2}$ then $G$ has a dominating path.
As $\delta(G) \geq k$ for $k$-connected graphs, this implies that for any $k > \sqrt{n-1}-1$, any $k$-connected graph has a dominating path.
Thus, Corollary \ref{cor:k-conn-dom}
improves Faudree's result for $n\leq 21$. Note that
Corollary \ref{cor:k-conn-dom} gave us a result for any $k$-connected graph, 
independent of its minimum degree.
%More recently, Faudree et. al.~\cite{FaudreeFGHJ18} showed 
%sufficient conditions for the existence of a dominating 
%path ($\pe(G) \leq 1$), depending on the sum of degrees of 
%independent vertices of $G$.

\section{Longest paths as central paths} \label{sec:longest}

In the previous section, we showed upper bounds 
for $\pe(G)$ based on the eccentricity of a longest path 
of a graph $G$. Since a longest path maximizes the number 
of vertices in it, intuitively it seems a good 
approximation for a central path of $G$. Moreover, 
in graphs that contain a Hamiltonian path, longest 
paths and central paths coincide. This led us to study 
the relationship between a longest path and a central path 
on a graph. In particular, we are interested in 
investigating which structural properties ensure the 
following.
\begin{property} \label{prop:longest}
	Every longest path is a central path of the graph.
\end{property}
In what follows, we tackle this question for different 
classes of graphs.

\subsection{Trees and planar graphs}

First, we consider the case in which the input graph 
is a tree. In this case, every longest path contains the
center(s) of a tree~\cite{ChartrandLZ16}. Cockayne et. al.~\cite{CockayneHH81} 
showed an analogous result for the case of minimum-length central paths.

\begin{theorem}[Cockayne et. al., 1981]
	Let $P^*$ be a central path of minimum 
	length in a tree $T$. Then~$P^*$ contains the center(s) of $T$. 
	Moreover, $P^*$ is unique.
\end{theorem}

Let $P$ be a longest path of $T$. Next, we show 
that $P$ contains $P^*$.

\begin{theorem}
	Let $T$ be a tree. Let $P$ be a longest path of $T$. 
	If $P^*$ is a central path of minimum length, then~$V(P^*) \subseteq V(P)$.
\end{theorem}
\begin{proof}
	Since $P$ and $P^*$ contain the center(s) of the tree, we have that 
	\[
		X := V(P) \cap V(P^*) \neq \emptyset.
	\]
	Furthermore, since $T$ is a tree, the set $X$ induces a subpath of $P$ and $P^*$, 
	say $X = \langle x_1, x_2, \ldots, x_k \rangle$. Suppose by contradiction 
	that $V(P^*) \nsubseteq V(P)$. Then, there exists an extreme of $P^*$, say $a$, that 
	is not in $P$. Let $u$ be an extreme of $P$ that is nearest to $a$. Without loss of 
	generality, suppose that $u$ is the closest extreme to~$x_1$. 

	By the minimality of $P^*$, there exists a vertex $a' \notin V(P^*)$ 
	such that 

	\begin{equation} \label{eq:minimal} 
		\begin{array}{rcl}
			d_T(a, a') & = & \ecc(P^*), \\
			d_T(b, a') & > & \ecc(P^*), \text{ for } b \in V(P^*) \setminus \{ a \}.
		\end{array}
	\end{equation}
	Moreover, observe that $a' \notin V(P)$, otherwise we have
	\[
		\min\{ d_T(x_1, a') , d_T(x_k, a') \} < d_T(a, a'), 
	\]
	which contradicts~$(\ref{eq:minimal})$.
	Thus, $a' \in V(T) \setminus (V(P^*) \cup V(P))$. Moreover, 
	the path linking $a$ to $a'$, in $T$, does not contain any vertex of $P$ or $P^*$ 
	(except $a$). Otherwise, there exists a vertex, in $P^*$, that is closer to $a'$ 
	than~$a$. Let $Q$ be the path between the vertices $x_1$ and $a'$. Also, 
	let $R$ be the path obtained from $P$ by replacing the path from $u$ 
	to $x_1$ by $Q$. Since $|R| \leq |P|$, we have that
	\[
		d_T(u, x_1) \geq d_T(a', x_1) > d_T(a, a') = \ecc(P^*) \geq d_T(u, x_1), 
	\]
	where the second inequality follows from~$(\ref{eq:minimal})$. Therefore, 
	we have a contradiction. \qed
\end{proof}

The previous result implies the following.

\begin{corollary}\label{cor:longest-tree}
	If $G$ is a tree, a longest path is a central path of $G$.
\end{corollary}

A natural direction is to extend the previous result for classes of 
graphs that contain trees. A minimal such superclass is \textit{cacti}. That is, 
the connected graphs where any two cycles have at most one vertex in 
common. As shown by Figure~\ref{fig:planar}~$(a)$, in this class, 
there are graphs such that a longest path is not a central path. In this figure, 
a longest path is composed of the vertices in the two cycles, which 
has eccentricity three, whereas the path eccentricity of the graph is two.

\begin{figure}[tbh]
	\centering
	\includegraphics[scale=.3]{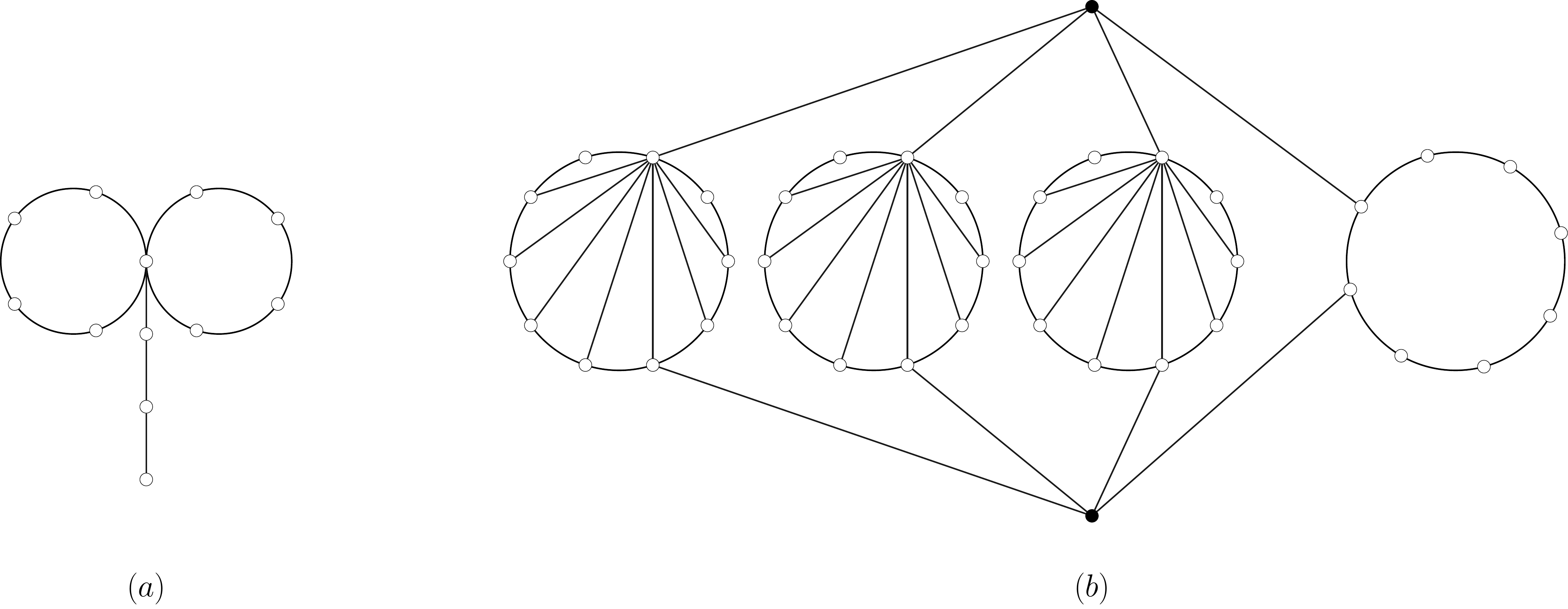}
	\caption{Two graphs that do not satisfy Property~\ref{prop:longest}: $(a)$ a cactus; 
	and $(b)$ a $2$-connected planar graph.}
	\label{fig:planar}
\end{figure}

Now, if we impose that a cactus is $2$-connected, then 
Property~\ref{prop:longest} holds, as these graphs are precisely cycles. 
Moreover, this additional constraint also ensures this property 
for a broader class of graphs. We say that a graph is \textit{outerplanar} if 
it can be embedded in the plane such that all its vertices belong to the outer 
face (and thus every cactus is outerplanar). Sys\l o~\cite{Syslo79} showed 
that $2$-connected outerplanar graphs are Hamiltonian, and therefore satisfy 
Property~\ref{prop:longest}. 

On the contrary, when we consider $2$-connected 
planar graphs, Property~\ref{prop:longest} does not hold, as shown by 
Figure~\ref{fig:planar}~$(b)$. In this figure, if we remove the black vertices, 
the resulting graph has four components: three having ten vertices and one having 
eight vertices. Clearly, every longest path is composed of the black vertices, 
and the vertices on the components with ten vertices. This implies that 
every longest path has eccentricity three. On the other hand, the path eccentricity 
of the graph is two. We can obtain such a path, by considering a path that contains 
the two black vertices, and the vertices that belong to the cycle of eight vertices. 
Tutte~\cite{Tutte56} showed that $4$-connected planar graphs are Hamiltonian, thus every 
longest path is central. Therefore, by imposing a higher vertex-connectivity, we can 
ensure Property~\ref{prop:longest} on planar graphs. We observe that, it is still open 
whether this property is also satisfied by $3$-connected planar graphs. 

\subsection{Classes with bounded path eccentricity}

In this section, we study Property~\ref{prop:longest} on classes 
of graphs that have bounded path eccentricity. 
A graph is \textit{AT-free} if for every triple of vertices,
there exists a pair of vertices in that triple, such that every path between them 
contains a neighbor of the other vertex of the triple.
Corneil et. al.~\cite{CorneilOL97} 
showed that every AT-free graph admits a dominating path.
A natural question is whether admitting a dominating path (or having bounded 
path eccentricity) implies that a longest path is also central. 
In what follows, we show a counterexample to this question. 

A graph $G$ is called an \textit{interval graph} if there exists 
a set of intervals (representing the vertices of the graph), in $\mathbb{R}$,
such that two vertices are adjacent if and only if its corresponding intervals 
intersect. Let $k \geq 3$ be an integer. Let $G$ be the interval graph 
of the following set of intervals.
\[
	\begin{array}{rcll}
		A_i & = & [1, 2^{k - i}], & \text{for } i = 1, \ldots, k, \\
		B_i & = & [2^{k - 1} + 1, 2^{k - 1} + 2^{k - i}], & \text{for } i = 1, \ldots, k, \\
		C_i & = & [2^k + i, 2^k + i + 1], & \text{for } i = 1, \ldots, k - 1, \\
		D & = & [1, 2^k + 1]. 
	\end{array}
\]
In Figure~\ref{fig:counterInt}, we show an example of $G$ for $k = 3$.
\begin{figure}
	\centering
	\includegraphics[scale=.6]{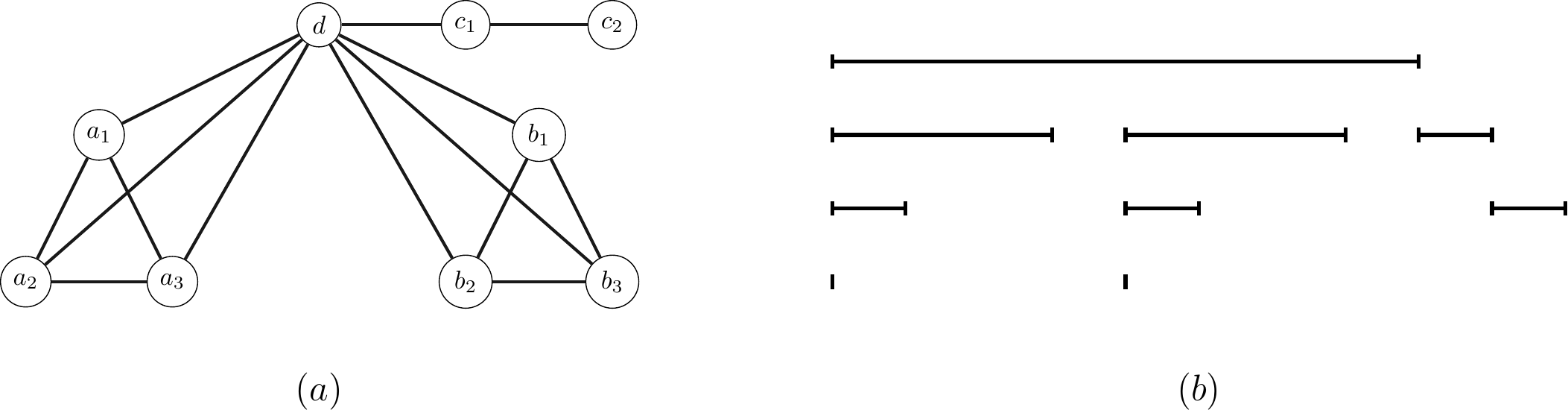}
	\caption{$(a)$ an interval graph $G$; and $(b)$ its interval representation.}
	\label{fig:counterInt}
\end{figure}

As $G$ is an interval graph and, thus an AT-free graph \cite{Lekkerkerker62}, we have 
that $\pe(G) \leq 1$. Let $a_i$ be the vertex representing interval $A_i$, 
for $i = 1, \ldots, k$. In a similar way, we define the vertices $b_i$, $c_i$ 
and $d$. Observe that $P = \langle a_k, a_{k - 1}, \ldots, a_1, d, b_1, \ldots, b_k \rangle$ 
is a longest path in $G$, and $\ecc_G(P) = k - 1$. Indeed, every longest 
path of $G$ has eccentricity $k - 1$. On the other hand, the 
path $Q = \langle a_k, \ldots, a_1, d, c_1, \ldots, c_{k - 1} \rangle$ 
has eccentricity one. Furthermore, $Q$ is a central path of $G$. 

We observe that, in the previous example, the interval representation 
of $G$ contains intervals that are properly included in other intervals. 
We say that a graph $G$ is a \textit{proper interval graph} if it admits 
an interval representation where no interval is properly contained in 
another interval. Bertossi~\cite{Bertossi83} showed that every connected 
proper interval graph has a Hamiltonian path (see Lemma 2). 
Therefore, we have the following.

\begin{corollary}\label{cor:longest-properinterval}
	If $G$ is a proper interval graph, a longest path is 
	a central path of $G$.
\end{corollary}

Now, we focus on bipartite graphs that have bounded path eccentricity. 
In Section~\ref{sec:bcg}, we showed that convex graphs have path 
eccentricity at most two. As shown by Figure~\ref{fig:cx}, convex graphs 
do not satisfy Property~\ref{prop:longest}. Let $G = (X \cup Y, E)$ denote the graph in 
this figure. The set $X$ is depicted by the white vertices. The labels on these 
vertices represent a convex ordering of $X$. Observe that, if we remove the 
vertex $s$, the resulting graph has three components, two of size eight and one of 
size four. Thus, every longest path of~$G$ is composed of $s$ and the vertices in 
the bigger components. Then, the eccentricity of any longest path, in~$G$, is four. 
On the other hand, $\pe(G) \leq 2$ by Corollary~\ref{cor:convexbound}. 

\begin{figure}[!bt]
	\centering
	\includegraphics[scale=.5]{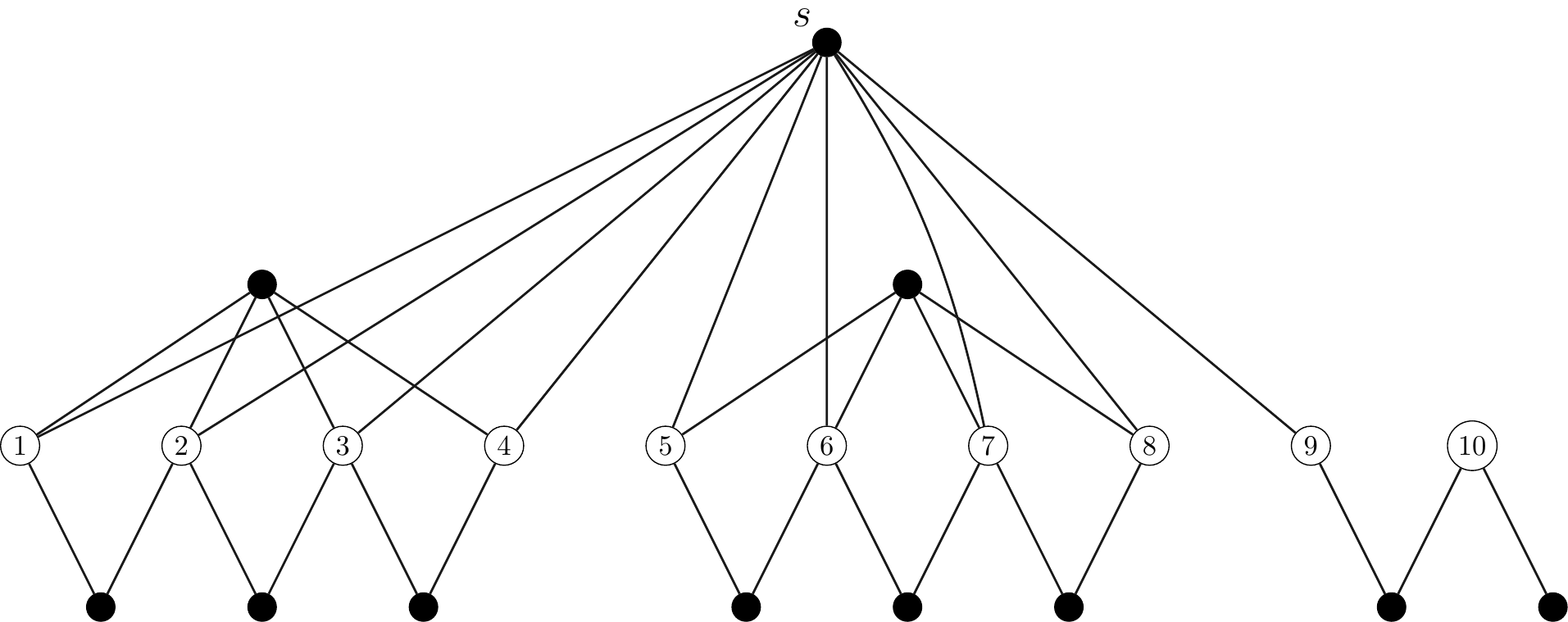}
	\caption{A convex graph that does not satisfy Property~\ref{prop:longest}.}
	\label{fig:cx}
\end{figure}

To conclude this section, we consider bipartite permutation graphs, a subclass of 
convex graphs. By Theorem~\ref{thm:dompath}, we have that $\pe(G) \leq 1$ for every 
bipartite permutation graph $G$. Now, suppose that~$G$ does not admit a Hamiltonian 
path. The following result from Cerioli et. al.~\cite{CerioliFGGL20} implies 
that $\ecc_G(P) = 1$, for every longest path $P$. 

\begin{lemma}[Cerioli et. al., 2018]
	Let $G$ be a bipartite permutation graph, and let $uv \in E(G)$. 
	Every longest path contains a vertex in $\{u, v\}$.
\end{lemma}

Therefore, Property~\ref{prop:longest} holds for this class.

\begin{corollary}\label{cor:longest-bpg}
	If $G$ is a bipartite permutation graph, a longest path is 
	a central path of $G$.
\end{corollary}

\section{Open problems and concluding remarks} \label{sec:remarks}

In this work, we studied the path eccentricity of graphs. First, we considered subclasses of 
perfect graphs. We showed that convex (resp. biconvex) graphs have path eccentricity at most 
two (resp. one). Moreover, we exhibit graphs that attain these bounds,
and design polynomial-time algorithms for finding such paths.
We observe that the adjacency matrix (or a submatrix of it) of a graph in 
these classes exhibit the \textit{consecutive ones property}~\cite{Golumbic04}. 
This property also appears in interval graphs (which have path eccentricity one). 
So, we believe the following question is important to understand the class of graphs 
that have bounded path eccentricity.
%this property (or variations 
%of it) regarding the path eccentricity of a graph, specially its relations with bounded 
%path eccentricity. 

\begin{question}
	How does the consecutive ones property (or variations of it) influence 
	the path eccentricity of a graph?
\end{question}

After that, we studied the path eccentricity of $k$-connected graphs, for $k \geq 1$. 
We showed a tight upper bound for this class. Interestingly, a graph that attains  
this bound has path eccentricity one. We posed the following question: there 
exists a $k$-connected graph, with path eccentricity $\ell$, that attain this bound? 
We give an affirmative answer for the cases: $k \leq 2$ and $\ell \geq 1$; $k = 3$ 
and $\ell = 1, 2, 3$; $k = 4$ and $\ell = 1, 2$; and~$k \geq 5$ and~$\ell = 1$. 
It is open whether such graph exists for $k = 3$ and $\ell \geq 4$; $k = 4$ 
and~$\ell \geq 3$; and~$k \geq 5$ and~$\ell \geq 2$. 
If such graphs do not exist, then the bound we obtained for $k$-connected graphs 
can be improved for $k \geq 3$. We consider interesting to study this 
question for $k = 3$.

\begin{question}
	Do there exists a better bound for $\pe(G)$ when $G$ is 3-connected?
\end{question}

To obtain the result for $k$-connected graphs, we studied the eccentricity of a longest path in a graph. 
It is natural to think that a path that contains the maximum number of vertices would 
also be a good candidate (or approximation) for a central path. For that 
reason, we investigated what structural properties imply 
that, in a graph, a longest path is also central. In what follows, we refer to this property 
as Property~\ref{prop:longest}. We showed that Property~\ref{prop:longest} is satisfied 
by trees. Furthermore, we consider subclasses of planar graphs that contain trees. 
We observe that, under an additional connectivity constraint, Property~\ref{prop:longest} 
is satisfied by cactus, outerplanar and planar graphs. This observation leads to the 
following question. Let $f$ be a function that, given a class of graphs, returns 
the minimum integer $k^*$ such that a $k^*$-connected graph, in this class, satisfies 
Property~\ref{prop:longest}. Note that, a trivial bound for this function follows 
from Ore's Theorem~\cite{BondyM08} with $k^* = (n-1)/2$, where $n$ is the number of 
vertices of the graph, since a graph $G$ with $\delta(G) \geq (n - 1)/2$ contains a Hamiltonian 
path. In particular, we exhibit tight bounds for this function on cactus and outerplanar graphs, 
and a constant upper bound for planar graphs.

\begin{question}
Investigate better upper bounds for $f$, in general, or in some classes of graphs. 
\end{question}

%We consider an interesting and challenging question 
%to investigate better upper bounds for this parameter, in general, or in some classes of graphs. 

\bibliographystyle{plain}
\bibliography{bibliografia}

\end{document}